\documentclass[12pt, twoside, leqno]{article}
\usepackage{amsmath,amsthm}
\usepackage{amssymb,latexsym}
\usepackage{enumerate}
\usepackage{amssymb,amsmath,amsthm,amsfonts}
\usepackage{mathrsfs}
\newcommand{\C}{\mathbb{C}}
\newcommand{\N}{\mathbb{N}}
\newcommand{\R}{\mathbb{R}}

\newtheorem{theorem}{Theorem}[section]

\newtheorem{lemma}[theorem]{Lemma}
\newtheorem{corollary}[theorem]{Corollary}

\theoremstyle{definition}
\newtheorem{definition}[theorem]{Definition}

\newtheorem{remark}[theorem]{Remark}

\numberwithin{equation}{section}
\begin{document}
\baselineskip=17pt

\author{Daniel Barlet \\ Institut Elie Cartan, Universit\'e de Nancy,\\ Laboratoire de Math\'ematiques\\ B.P. 239, 54506 Vandoeuvre l\`es Nancy Cedex, France\\
E-mail: Daniel.Barlet@iecn.u-nancy.fr
\and Teresa Monteiro Fernandes\\ Centro de Matem\'atica e Aplica\c c\~oes Fundamentais\\ Departamento de Matem\'atica\\ da Faculdade de Ci\^encias da Universidade de Lisboa\\Edif\'icio C6, P.2, Campo Grande,\\ 1749-16 Lisboa, Portugal\\
E-mail: tmf@ptmat.fc.ul.pt}

\title{Grauert's theorem for subanalytic open sets in real analytic manifolds}
\date{}
\maketitle
\renewcommand{\thefootnote}{}



\renewcommand{\thefootnote}{\arabic{footnote}}
\setcounter{footnote}{0}

\newpage

\begin{abstract}
\noindent By  open neighbourhood  of an open subset $\Omega$ of $\mathbb{R}^n$ we mean an open subset $\Omega'$ of $\mathbb{C}^n$ such that $\mathbb{R}^n\cap\Omega'=\Omega.$ 
A well known result of H. Grauert 
  implies  that any  open subset of $\mathbb{R}^n$ admits a fundamental system of Stein open neighbourhoods in $\mathbb{C}^n$. Another way to state this property is to say that each open subset of $\mathbb{R}^n$ is Stein.\\
\noindent We shall prove a similar result in the subanalytic category, so, under the assumption that $\Omega$ is a subanalytic  open subset  in a paracompact real analytic manifold, 
we show that \ $\Omega$ \ admits a fundamental system of subanalytic Stein open neighbourhoods in any of its complexifications.
\end{abstract}

\section{Introduction.}

A classical result of H. Grauert gives that an open set in a real analytic manifold \ $M_{\mathbb{R}}$ \ is locally the trace on \ $M_{\mathbb{R}}$ \ of a Stein open set in any given complexification \ $M_{\mathbb{C}}$ \ of \ $M_{\mathbb{R}}$.\\
The analogous result in the semi-analytic setting is easy to obtain because when \ $f$ \ is a real analytic function, say near \ $0$ \ in \ $\mathbb{R}^n$, the set \ $\{ f > 0 \}$ \  is  near \ $0$ \ the trace on \ $\mathbb{R}^n$ \ on the Stein open set \ $\{\Re(f) > 0 \} $ \ cut with a small open ball in \ $\C^n$.\\
We solve the subanalytic case of this problem using the rather deep following result (Theorem \ref{L:1} below):
\begin{itemize}
\item a compact subanalytic set in \ $\mathbb{R}^n$ \ may be defined as  the zero set of a \ $\mathscr{C}^2$ \ subanalytic function on \ $\mathbb{R}^n$. 
\end{itemize}

\noindent The construction of the subanalytic Stein open subset we are looking for is then an easy consequence of H. Grauert's idea.

\noindent Let us mention without technical details  that applications of our result arise naturally in the theory of sheaves on subanalytic sites, as it has been developped by L. Prelli in \cite{LP} (cf. \cite{KS} for the foundations of the theory of ind-sheaves). It entails, for instance, that the subanalytic sheaf of tempered analytic functions on a real analytic manifold is concentrated in degree zero as in the classical case.

 \noindent We conclude this article by computing one very simple example which is
not semi-analytic in order to show that the subanalytic case is much more
involved and also to explain to non specialists of subanalytic geometry (as we are) what are the ideas and tools hidden behind this construction.\\

\noindent We wish to thank Adam Parusinski for having pointed out to us a precise reference of Theorem \ref{L:1}, and the referee for asking us about the  unbounded case.

\section{Main results and proofs}

We refer to\ \cite{BM}, \cite{MC}, \cite{KK} and  \cite{S}  for the basic material on subanalytic geometry.

The following result is due to Bierstone, Milman and Pawlucki in a private letter to W. Schmid and K.Vilonen in 1995  (cf.  \cite{SV}). We refer \cite{DM}, C.11, for a proof in the more general setting of o-minimal structures. 

\begin{theorem}\label{L:1}

Let  \ $A$ \  be a compact subanalytic set of \ $\mathbb{R}^n$ \ and let  \ $p\in\mathbb{N}$ \  be given. Then there exists a \  $\mathscr{C}^p$ subanalytic function $f$ in \  $\mathbb{R}^n$ \  such that \ $A=f^{-1}(0)$.
\end{theorem}

\begin{remark} Let \ $U$ \ be a open ball in \ $\R^n$ \ and \ $Z $ \ a relatively compact subanalytic open set in \ $U$. Then there exists a \ $\mathscr{C}^2$ \ subanalytic function \ $g : \R^n \to \R^+$ \ with compact support in \ $U$ \ such that 
 $$ Z = \{ x \in \R^n \ ; \ g(x) > 0 \} .$$
  Apply the previous theorem to \ $\bar U \setminus Z$  \ and define \ $g$ \ to be \ $f$ \ on \ $U$ \ and \ $0$ \ on \ $\R^n \setminus U$. As \ $U$ \ is subanalytic and \ $f$ \ identically zero around \ $\partial U$, this function \ $g$ \ satifies the required properties.\\
  Moreover, we can divide this function \ $g$ \ by any given positive constant without changing the set \ $Z$, so for each \ $\varepsilon > 0$ \ we may assume that the Levi form of \ $g$ \ is uniformely bounded on \ $\R^n$ \ by \ $\varepsilon.\vert\vert z \vert\vert^2$.
\end{remark}

%Using a\ $\mathscr{C}^2$  \ subanalytic partition ( \cite{KK}, \cite{K} for  detailed proofs of the existence, for any \ $p$, of subanalytic  \ $\mathscr{C}^p$ partitions of the unity),  gives the following corollary: 

\begin{corollary}\label{P:123}
Let \ $\Omega$  be a  subanalytic open set in a real paracompact  analytic manifold \ $M_{\R}$. Then, for any complexification \ $M_{\C}$\ of \ $M_{\R}$, and for any given smooth  hermitian metric on the  complex tangent bundle on \ $M_{\C}$ \   there exists a subanalytic non negative real function \ $f$ on \ $M_{\C}$\ of class \ $\mathscr{C}^2$ such that  
$$\{f > 0\}\cap M_{\R} = \Omega$$
and such that the Levi form of \ $f$ \  is bounded by the given hermitian metric.\\
Moreover, \ $f$\ can be chosen so that \ $Supp f$\  is contained in any given open set in \ $M^{\C}$\ containing 
the closed set \ $\bar{\Omega}$.

\end{corollary}
\begin{proof}

For \ $\epsilon>0$, let us denote \ $B_{\epsilon}$\  an open ball of \ $\mathbb{R}^n$ \ of radius \ $\epsilon$\ and  by \ $B_{\epsilon}^{\C}$\  the corresponding ball  in  \ $\C^n$.

 \noindent For each \ $p\in\bar{\Omega}$\ (the closure of \ $\Omega$)  there exists two relatively  compact open subanalytic neighbourhoods \ $V \subset\subset U$ \ of \ $p$ \ in \ $M_{\C}$\ and a complex analytic isomorphism \ $\varphi$ \ defined in an open  neighbourhood $W$ of  \ $\bar U$ \ to an open ball \ $B^{\C}_{\epsilon}$\, such that \ $\varphi(\bar V)$ \ is the closed ball $\bar{B}^{\C}_{\epsilon/2}$, and  \ $\varphi$ is real  on \ $W\cap M_{\mathbb{R}}$. In particular, \ $\bar V\cap M_{\mathbb{R}} \subset U$\  is a compact  subanalytic subset, and \ $\bar U$ \ is a compact subanalytic subset of \ $W$.\\
As \ $M_R$ \ is paracompact, we get a locally finite  countable cover \   $  (U_i)_{i \in \mathbb{N}^*}$ \  of \ $\bar \Omega$ \ such that  the conditions above are satified.  On each \ $U_i$, by the remark following the Theorem \ref{L:1}, we may choose a \ $\mathscr{C}^2$ non negative subanalytic function \ $f_i$\ on \ $M_{\C}$ \   with compact support  in \ $U_i$ \ whose non  zero set is exactly \ $V_i\cap \Omega$, and such that its Levi form  is bounded by \ $h/2^i$ \ for any given hermitian metric \ $h$ \ on \ $M_{\C}$.  Then   define \ $f : = \sum_{i=1}^{\infty} \ f_i$. As this sum is locally finite, it clearly satisfies our requirements.\\
The last assertion follows by applying this construction in any open neighbourhood \ $W$ \ of \ $\bar \Omega$ \ in \ $M_{\C}$ \ regarded as a complexification of \ $W \cap M_{\R}$. \end{proof}

%\begin{remark}\quad \\
%We stated and proved the previous corollary in the exact form we shall use. However, this corollary may be deduced from the fact that, with the same argument,  Theorem \ref{L:1} is true when replacing \ $\mathbb{R}^n$ \ by  any paracompact real analytic manifold.
%\end{remark}

\bigskip

\begin{theorem}\label{T:2}
Let \ $\Omega$ \ be a subanalytic open set of a real paracompact analytic  manifold \ $M_{\mathbb{R}}$. Then, given a complexification\  $M_{\C}$\ of  \ $M_{\mathbb{R}}$, there exists a 
 \ subanalytic Stein  open subset \ $\Omega_{\C}$\ of \ $M_{\C}$ \, such that   
\begin{equation}\label{E:1}
\Omega=\Omega_{\C} \cap \, M_{\mathbb{R}}
\end{equation}
\end{theorem}
\begin{proof}
Let \ $n$ be the dimension of \ $M_{\R}$. By Grauert's Theorem 3, page \ $470$\ of \cite{G}, there exist a natural number \ $N\in\N$\ and a real  analytic regular proper embedding \ $\varphi$ of \ $M_{\R}$ \ in the euclidean space \ $\R^N$. By complexification,  one defines a holomorphic map \ $\varphi_{\C}$\ in a neighborhood \ $V$\ of \ $M_{\R}$\ in \ $M_{\C}$  taking values in \ $\C^N$,  such that \ $\varphi_{\C}|_{M_{\R}}=\varphi$ \  and such that the rank of \ $\varphi_{\C}$ is everywhere equal to \ $n$.

%%%%%%%%%%%%%%%%%%%%%%%%%%
%%%%%%%%%%%%%%%%%%%%%%%%%%%%%%%%%%%%

%%%%%%%%%%%%%%%%%%%%%%%%%%
%%%%%%%%%%%%%%%%%%%%%%%%%%%%%%%%%%%%

 Note that  the Levi form of the real analytic function 
  $$ g(z_1,..., z_N)= \sum_{j=1}^N \quad  ({\Im}\ z_j)^2  $$
   is half the square norm in \ $\C^N$, hence \ $g$ is strictly plurisubharmonic on \ $\C^N$. By the maximality of the rank of \ $\varphi_{\C}$,  the function $\varphi_{\C}^*(g)$  is  also strictly plurisubharmonic on \ $V$ and subanalytic (in fact analytic). \\
   Fix now a smooth hermitian metric\footnote{for instance 1/2 of the Levi form of \ $\varphi_{\C}^*(g)$  \ may be choose as K{\"a}hler form on \ $V$.} \ $h$ \ on \ $T_{\C}V$ \ such that  the Levi form of \ $\varphi_{\C}^*(g)$  \ is bigger at each point than \ $2.h$.\\
By Proposition \ref{P:123}, there exists a subanalytic \ $\mathscr{C}^2$\ non negative function $f$ with  support on \ $V$ such that 
$$\{f>0\} \cap M_{\R}=\Omega $$
and such that the Levi form of \ $f$ \ is bounded by \ $h$. So the Levi form of the \ $\mathscr{C}^2$\ subanalytic function 
 $$ \psi : = \varphi_{\C}^{*}(g)-f$$
  is positive definite at each point of \ $V$. It follows that the open set 
 $$\Omega_{\C}=\{\psi <0\} \cap V$$ 
 is (strongly 1-complete)  Stein by Grauert's famous result and subanalytic  in \ $M_{\C}$ \ by construction. 
 
 Moreover, as we have \ $\varphi_{\C}^{*}(g) = 0 $ \   in \ $ M_{\R}$, it follows that \ $\Omega_{\C} \cap M_{\R} = \Omega$. 
 \end{proof}

%In particular we get:

%\begin{corollary}\label{T:1}
%Let \ $\Omega$\ be a relatively compact subanalytic open set of \ $\mathbb{R}^n$. Then, there exists a Stein subanalytic open set  \ $\Omega_{\C}$ \  in  \ $\C^n$ \  such that \  $$\Omega=\mathbb{R}^n\cap\Omega_{\C}.$$
%\end{corollary}

%%%%%%%%%%%%%%%%%%%%%%%%%%
%%%%%%%%%%%%%%%%%%%%%%%%%%%

%\begin{remark}
%The referee asks if we cannot get the preceding theorem without relative compactness. I think he is right, I would try to give a rigorous proof from what I quickly write below:

%A subanalytic open subset $\Omega$ in a paracompact real analytic manifold $M$ admits a countable exhaustion by subanalytic open relatively compact subsets \ $\Omega_i$, that is, \ $\Omega=\cup_{i\in\N} \Omega_i$ and, for every \ $i$,  \ $\bar{\Omega_i}\subset \Omega_{i+1}$. (am I wrong?). For each $i$, we have $$\Omega_i=\tilde{\Omega_i}\cap M$$ with $\tilde{\Omega}_i$ Stein subanalytic
%in $M^{\C}$, and we may assume, up to replacing $\tilde{\Omega}_{i-1}$ by the intersection  $\tilde{\Omega}_i\cap\tilde{\Omega}_{i-1}$, that the family $(\tilde{\Omega}_i)$ is a Stein exhaustion of $\cup_{i\in\N}\tilde{\Omega}_i$ which is a Stein open subset  subanalytic (since it is a locally finite union of subanalytic sets) (???) of $M^{\C}$.
%\end{remark}

%%%%%%%%%%%%%%%%%%%%%%%%%%%%%%%%%
%%%%%%%%%%%%%%%%%%%%%%%%%%%%%%%%%%%%%%
\section{Example:  
A strange four-leaved trefoil}

Our aim is now to give an explicit construction of the function \ $f$ \ in Theorem \ref{L:1} in the case of one of the simplest example which is not semi-analytic. For that purpose we shall only use {\L}ojasiewicz inequalities and Theorem \ref{deriv.} which are basic tools in subanalytic geometry. We think that this analysis will convince the reader of the strength and usefulness of Theorem \ref{L:1} and  that this tool is far from being elementary.\\

 \noindent We shall need the following refinement of subanalyticity.

\subsection{Strong subanalyticity}
For a continuous function \ $f : \mathbb{R}^n \to \mathbb{R}$ \ to be subanalytic simply means that its graph is a subanalytic set in \ $\mathbb{R}^n \times \mathbb{R}$, but in the non continuous case we shall use a stronger assumption, in order to control the behaviour of the graph near points where \ $f$ \ is not locally bounded. We restrict ourself to the context of the situation we need here.

\begin{definition}\label{strongly sub.}
Let \ $\Omega \subset\subset \mathbb{R}^n$ \ a relatively compact subanalytic open set, and let
$$ f : \Omega \to \mathbb{R} $$
be a continuous function. We shall say that \ $f$ \ is {\bf strongly subanalytic} if the function \ $\tilde{f} : \mathbb{R}^n \to \mathbb{R} $ \ defined by extending \ $f$ \ by \ $0$ \ on \ $\mathbb{R}^n \setminus \Omega$ \ has a subanalytic graph in \ $\mathbb{R}^n \times \mathbb{P}_1$, where \ $\mathbb{P}_1$ \ is the \ $1-$dimensional projective space \ $\mathbb{R} \cup \{\infty\}$.
\end{definition}

\noindent It is easy to see that such a condition implies that the growth of \ $f$ \ near a boundary point in \ $\partial \Omega$ \ has to be bounded by some power of the function \ $d(x, \partial \Omega)$ \ thanks to {\L}ojasiewicz inequalities (\cite{BM}).\\
Remark that if \ $\tilde{f}$ \ is continuous this condition reduces to the usual subanalyticity of the graph of \ $\tilde{f}$ \  in \ $\mathbb{R}^n\times \mathbb{R}$.\\

\noindent We shall need also the following theorem (cf.\cite{KK},  Theorem (2.4)).

\begin{theorem}\label{deriv.}
Let \ $\Omega \subset\subset \mathbb{R}^n$ \ a relatively compact subanalytic open set, and let
$$ f : \Omega \to \mathbb{R} $$
be a \ $\mathscr{C}^1$ \   function which is strongly subanalytic. Then any partial derivative of \ $f$ \ in \ $\Omega$ \ is also strongly subanalytic.
\end{theorem}

\noindent Since, in Definition \ref{strongly sub.},  the continuity of \ $\tilde{f}$ \ just means that \ $f(x)$ \ goes to \ $0$ \ when  \ $x \in \Omega$ \ goes to  the boundary \ $\partial \Omega$, using {\L}ojasiewicz inequalities we easily obtain the following corollary:

\begin{corollary}\label{C1}
 In the situation of the previous theorem, assume that \ $\tilde{f}$ \ is continuous. Then there exists an integer \ $N_1$ \ such that \ $\tilde{f}^{N_1}$ \ is \ $\mathscr{C}^1$ \ on \ $\mathbb{R}^n$ \ and subanalytic.
\end{corollary}

\noindent Now applying again the ideas of the previous  corollary  we finally obtain:

\begin{corollary}\label{C2}
In the situation of the previous corollary there exists an integer \ $N_2$ \ such that \ $\tilde{f}^{N_2}$ \ is \ $\mathscr{C}^2$ \ on \ $\mathbb{R}^n$ \ and subanalytic.
\end{corollary}

\noindent \begin{remark}As the reader can see in view of the preceding results, the remaining and non trivial step to prove the existence of a subanalytic \ $\mathscr{C}^2$ \ function which vanishes exactly on \ $\mathbb{R}^n \setminus \Omega$ \ as stated in Theorem \ref{L:1}, is to show the existence of a \ $\mathscr{C}^2$ \ strictly positive  (strongly) subanalytic function \ $f$ \ on \ $\Omega$ \ which vanishes at the boundary. The natural candidate is, of course, the function \ $x \mapsto d(x, \partial \Omega)$. But all conditions are satisfied excepted smoothness. And the non smoothness points may go to the boundary.
If one tries to use the "desingularization theorem" of H. Hironaka to solve this problem, a new difficulty comes  then because the jacobian of the modification may vanish inside \ $\Omega$ \ and not only on some points in \ $\partial \Omega$.
\end{remark}

\subsection{Example}
Let us consider the analytic map
\ $ F : \mathbb{R}^3 \to \mathbb{R}^3 $ \ 
defined by
$$ F(x,y,z) = \big( y.(e^x - 1) + x^2 + y^2 + z^2 - \varepsilon^2, y.(e^{x.\sqrt{2}} - 1), y.(e^{x.\sqrt{3}}-1)\big).$$
Denote  \ $\Omega$ \ the interior of the image \ $\tilde{\Omega}$ \ of the compact ball \ $\bar B_3(0,\varepsilon)$. Let us start by showing that the image by \ $F$ \ of the sphere  \ $S_{\varepsilon}$ (the boundary  of  \ $\bar B(0,\varepsilon)$), \ is a subanalytic compact subset of \ $\mathbb{R}^3$ \ which is not semi-analytic in the neighborhood of \ $(0,0,0)$. This example is extracted from \cite{HH2}( example I.6).

\begin{lemma}\label{sous-ana.}
The compact \ $F(S_{\varepsilon})$ \ is not semi-analytic in the neighbourhood of the origin.
\end{lemma}

\begin{proof} Since this compact set has an empty interior, if it is semi-analytic in a neighbourhood of the origin, there shall exist an analytic function \ $f : U \to \mathbb{R}$ \ on a ball \ $U$ \ centered in \ $0$, non identically zero, such that \ $f^{-1}(0)$ \ contains \ $U \cap F(S_{\varepsilon})$. Let 
$$ f = \sum_{m \geq m_0} \ P_m $$
be the Taylor series of \ $f$ at the origin, which we may assume to be convergent in \ $U$ \  provided that  \ $U$ is small enough. We shall assume that the homogeneous polynomial $P_{m_0}$ \ is not identically zero. Hence, considering \ $(x,y,z) \in S_{\varepsilon}$ \ close enough to \ $(0,0,\varepsilon)$, the definition of \ $F$ \ entails  the equality
$$ 0 \equiv  \sum_{m \geq m_0} \ y^m.P_m((e^x - 1),(e^{x.\sqrt{2}} - 1),(e^{x.\sqrt{3}} - 1)) $$
which holds  for \ $(x,y) \in \mathbb{R}^2$ \ close enough to \ $(0,0)$. We conclude that \ $P_{m_0}((e^x - 1),(e^{x.\sqrt{2}} - 1),(e^{x.\sqrt{3}} - 1)) $ \ is identically zero for \ $x$ \ in a neighbourhood of \ $0$. Hence this analytic function vanishes identically on \ $\mathbb{R}$. The behaviour at infinity of this function easily entails\footnote{ This is equivalent to prove the algebraic independency  of the functions \\ $(e^x - 1),(e^{x.\sqrt{2}} - 1),(e^{x.\sqrt{3}} - 1)$.} that we must have \ $P_{m_0} \equiv 0$, which gives a contradiction. 
\end{proof}
We shall now describe the open set \ $\Omega$. Let us remark that the jacobian of  \ $F$ \ is given by
$$ J(F)(x,y,z) = 2yz.\big((\sqrt{2} -\sqrt{3}).e^{x.(\sqrt{2}+\sqrt{3})} - \sqrt{2}.e^{x.\sqrt{2}} + \sqrt{3}.e^{x.\sqrt{3}}\big) $$
and for \ $\varepsilon$ \ small enough, it doesn't vanish on \ $\{ x.y.z = 0 \}$ \ within the ball  \ $\bar B_3(0, \varepsilon)$.  Indeed, the brackets give an analytic function of a single variable \ $x$; hence it has an isolated  zero in \ $x = 0$. The image of \ $\{ x.y = 0 \} \cap \bar B_3(0,\varepsilon)$ \ by \ $F$ \ is  \ $[-\varepsilon^2,0]\times \{(0,0)\}$ \ which is contained in \footnote{  See the description of \ $\Gamma$ near \ $(0,0)$ given below}  the boundary of \ $\tilde{\Omega}$.\\
The image of \ $\{z = 0\}$ \ is more complicated to describe.\\

Let us now consider the analytic morphism \ $G : \mathbb{R}^2 \to \mathbb{R}^2 $ \ defined by $$ G(x,y) : = \big( y.(e^{x.\sqrt{2}} - 1), y.(e^{x.\sqrt{3}}-1)\big).$$
Let us denote by  \ $\Gamma$ \ the image by \ $G$ \ of the ball \ $\bar B_2(0,\varepsilon)$ \ of \ $\mathbb{R}^2$. 
If \ $(v,w) \in \Gamma \setminus \{(0,0)\}$ \ then the fiber \ $G^{-1}(v,w)$ \ is reduced to a single point (for \ $\varepsilon$ \ small enough). In fact we must have \ $v.w \not= 0$ \ and
 $$ \frac{(e^{x.\sqrt{2}} - 1)}{(e^{x.\sqrt{3}} - 1)} = \frac{v}{w} = \frac{\sqrt{2}}{\sqrt{3}}.h(x) $$
 whenever \ $h \in \C\{x\}$ \ converges for \ $\vert x\vert < 2\pi/\sqrt{3}$ \ et verifies \ $h(0) = 1$ \ and \ $ h'(0) = (\sqrt{2} - \sqrt{3})/2$;  these equations determine a unique  \ $x \in [-\varepsilon,\varepsilon]$, for \ $\varepsilon \ll1 $, and hence a unique \ $y$. Remark that for  \ $x$ \ in a neighbourhood of \ $0$, we have \ $v/w$ \ close  to \ $\sqrt{2}/\sqrt{3}$. Therefore \ $\Gamma$ \ doesn't approach \ $(0,0)$ \ other than along that direction.\\
The fiber in \ $(0,0)$ \ of \ $G$ \ is the curve  \ $\{ x.y = 0 \} \cap \bar B_2(0,\varepsilon)$.\\
 
 Remark that the points in the sphere \ $ \{ x^2 + y^2 = \varepsilon^2 \}$ \ are mapped  on the boundary of \ $\Gamma$. Indeed, for those who lie on \ $\{ x.y = 0 \}$ \ their image is the origin. Otherwise, for each of such points not mapped on the origin, the jacobian of \ $G$ would \ vanish and the boundary of \ $\bar B_2(0,\varepsilon)$ \ would be mapped on the boundary of \ $\Gamma$ \ in its neighbourhood.\\
  Hence, any point of the interior \ $\Gamma'$ \ of \ $\Gamma$ \ is the image by \ $G$ \ of some point in \ $B_2(0,\varepsilon) \setminus \{x.y = 0 \}$.\\

  We shall denote by \ $\varphi : \Gamma \setminus \{(0,0)\} \to \mathbb{R} $ \ the subanalytic function \footnote{ The graph of \ $G^{-1} : \Gamma\setminus \{(0,0)\} \to \bar B_2(0,\varepsilon \setminus \{ x.y = 0\}$ \ is the same as that  the graph of \ $G :  \bar B_2(0,\varepsilon) \setminus \{ x.y = 0\} \to \Gamma \setminus \{(0,0)\} $.} 
 given by \ $\varphi(v,w) = \vert\vert G^{-1}(v,w)\vert\vert^2$, in other words, the composition of \ $G^{-1}$ \ with the square of the euclidean norm in \ $\mathbb{R}^2$.\\
 We shall denote by \ $\psi : \Gamma \setminus \{(0,0)\} \to \mathbb{R} $ \ the subanalytic function defined by setting \ $\psi(v,w) = y.(e^x -1)$ \  where \ $G^{-1}(v,w) = (x,y)$, and we set
 
  \begin{align*}
 & \Delta^+ : = \big\{(\psi(v,w), v, w), \quad {\rm for} \quad (v,w) \in \Gamma \setminus \{(0,0)\}\big\} \\
 & \Delta^- : = \big\{ (\psi(v,w) + \varphi(v,w) -\varepsilon^2, v, w), \quad {\rm for} \quad (v,w) \in \Gamma \setminus \{(0,0)\}\big\} \\
 & \Delta^0 : = [-\varepsilon^2,0]\times \{(0,0)\}\  
 \end{align*}
 
 Note that  $$\Delta^+ \cap \Delta^- = \big\{ (u,v,w) \in \mathbb{R} \times (\Gamma \setminus \{(0,0)\}) \ / \ u = \psi(v,w) \quad{\rm and} \quad \varphi(v,w) = \varepsilon^2 \big\}$$
 is  the graph of the restriction of \ $\psi$ \  to \ $\partial\Gamma \setminus \{(0,0)\}$.\\
 
 We have now the following description of  \ $\tilde{\Omega}$ \ and of its interior \ $\Omega$.
 
 \begin{lemma}\label{L:24}
 One has \ $ \partial\tilde{\Omega}  = \Delta^+ \cup \Delta^- \cup \Delta^0$.
 The interior \ $\Omega$ \ is the open set
 $$ \Omega = \big\{ (u,v,w) \in \mathbb{R}\times \Gamma' \ / \   \psi(v,w) + \varphi(v,w) -\varepsilon^2  <  u  < \psi(v,w) \big\} $$
 where\ $\Gamma'$ \ denotes the interior of  \ $\Gamma$.
 \end{lemma}

\begin{proof} Let\ $(u,v,w) \in \tilde{\Omega}$. If \ $v.w= 0$ \ then \ $x.y = 0$ \ and \ $v = w = 0$, and \ $u = x^2 + y^2 + z^2 - \varepsilon^2 $ \ belongs to  \ $[-\varepsilon^2,0]$ which is contained in \ $\Delta^0$. Since the projection of \ $\Omega$ \ on \ $\mathbb{R}^2$ \ is an open set contained in \ $\Gamma$, hence in  \ $\Gamma'$, the point \ $(u,v,w)$ does not belong to \ $\Omega$. Let us now exclude this case.\\

  We have a point \ $(x,y,z) \in \bar B_3(0, \varepsilon)$ \ such that \ $F(x,y,z) = (u,v,w)$, with \ $x.y \not= 0$. Then \ $(x,y) \in \bar B_2(0, \varepsilon) \setminus \{x.y = 0\}$ \ and \ $G(x,y) = (v,w) $ \ is not\ $(0,0)$. Since \ $\varphi(v,w) = x^2 + y^2$ we have
  $$ u = \psi(v,w) + \varphi(v,w) + z^2 - \varepsilon^2 $$
  where \ $z \in [-\varepsilon,\varepsilon]$ \ is, up to a sign, determined by this equation. We conclude that the inequalities
  \begin{equation}\label{E:10}
   \psi(v,w) + \varphi(v,w) -\varepsilon^2  \leq u \leq \psi(v,w) 
   \end{equation}
  hold  on \ $\tilde{\Omega}$. 
We have to check  that   \ $\partial\tilde{\Omega}\setminus\Delta^0$ \ is exactly described by the equality
   \begin{equation}\label{E:11}
   (u - \psi(v,w) - \varphi(v,w) + \varepsilon^2)(\psi(v,w) - u) = 0 . 
   \end{equation}
   Since the projection on \ $\mathbb{R}^2$ \ is open,  \ if\ $(v,w) \not\in \Gamma'$ then it must lie in the boundary of  \ $\Omega$. It suffices to prove that for \ $(v,w) \in \Gamma'$ \ the equality above implies that \ $(v,w)$\ is in the boundary. This is clear because near any \ $(u,v,w)$ \  of \ $\Omega$ \ one can find \ $\delta >0$ \ such that \ $]u-\delta,u+\delta[\times (v,w)$ \ is contained in \ $\Omega$, which is not possible by the inequalities (\ref{E:10}) in a point where the equality (\ref{E:11}) is satisfied.\\
   Hence it is sufficient to prove that \ $\tilde{\Omega} \setminus \Delta^0$ \ is the set of points \ 
   $(u,v,w) $ \ in \ $ \mathbb{R}\times (\Gamma\setminus\{(0,0)\})$ \ satisfying the inequalities (\ref{E:10}). Indeed, any  choice of \ $(v,w) \in \Gamma\setminus\{(0,0)\}$ \ gives a unique point  \ $(x,y) \in B_2(0,\varepsilon)$ \ such that \ $G(x,y) = (v,w)$ \ and the inequalities (\ref{E:10}) \ entail that we can find at least a \ $z \in \mathbb{R}$ \ such that \ $z^2 = u - \psi(v,w) - \varphi(v,w) + \varepsilon^2$ \ and that \ $\varphi(v,w) + z^2 \leq \varepsilon^2$.   \\
   Note that if \ $u = \psi(v,w) + \varphi(v.w) - \varepsilon^2$ \ we will have \ $z = 0$. Therefore, the boundary \ $\Delta^-$ \ corresponds to the image of \ $\bar B_3(0,\varepsilon) \cap \{z = 0\}\setminus\Delta^0$.
   Similarly the equality \ $u = \psi(v,w)$ \ corresponds to the image of the sphere \ $\{x^2 + y^2 + z^2 = \varepsilon^2\}$ \ deprived of\ $\Delta^0$ . \hfill    \end{proof}
  \bigskip
  
  Let us now consider the function \ $f : \mathbb{R}^3 \to \mathbb{R}^+ $ \ defined as follows:
  \begin{itemize}
  \item For \ $(u,v,w) \in \Omega$ \ one sets \\ 
  $f(u,v,w) =  ( \psi(v,w) - u)(u - \psi(v,w) - \varphi(v,w) + \varepsilon^2) $
  \item For \ $(u,v,w) \not\in \Omega$ \ one sets \ $f(u,v,w) = 0 $.
  \end{itemize}
  Note that  $f$ \ est strictly positive on \ $\Omega$ \ by Lemma \ref{L:24}, and that it is  analytic on the complement of \ $\partial \Omega$, since the functions \ $\varphi$ \ and \ $\psi$ \ are analytic  on \ $\Gamma'$. Moreover \ $f$ \ is bounded.
    
    Let us now define \ $\tilde{f}(u,v,w) = f(u,v,w).v^2.w^2 $. 
    
  \begin{lemma}\label{L:25}
  The function \ $\tilde{f} : \mathbb{R}^3 \to \mathbb{R}^+ $ \ is  subanalytic and continuous,  it satisfies
  $$ \Omega  = \{ (u,v,w) \in \mathbb{R}^3 \ / \  \tilde{f}(u,v,w) > 0 \} $$
  and it is  \ $\mathscr{C}^{\infty}$ \ on \ $\mathbb{R}^3 \setminus \partial \Omega$.
  \end{lemma}
  
  \begin{proof} First we prove that \ $f$ \ is subanalytic\footnote{As pointed by the referee, this fact is consequence of basic stability properties of subanalytic functions. We give a direct proof for non specialists.}. Since its graph is the union of the graph  of its restriction to \ $\Omega$ \ and the set \ $(\mathbb{R}^3\setminus \Omega)\times \{0\}$\ which is subanalytic,  $\Omega$ \ being an open subanalytic set  of \ $\mathbb{R}^3$, it is sufficient to prove that the graph of the restriction of \ $f$ \ to \ $\Omega$ \ is subanalytic. \\
 Let us consider the polynomial  morphism\ $h : \mathbb{R}^3 \to \mathbb{R}$ \ given by
  $$ h(x,y,z) = (\varepsilon^2 -(x^2 + y^2 + z^2) ).z^2 $$
  and denote by \ $X, X_1, X_2$ \ the  graph of the restriction of \ $h$ \ respectively to \\
   $\bar B_3(0,\varepsilon), \partial B_3(0,\varepsilon), \bar B_3(0,\varepsilon) \cap \{x.y = 0 \}$ \   and \ $Y, Y_1, Y_2$ \ the respective images  of these graphs by the morphism \ $F \times id : \mathbb{R}^3 \times \mathbb{R} \to \mathbb{R}^3 \times \mathbb{R}$. 
   
  Let us prove that the graph of the restriction of \ $f$ \  to \ $\Omega$ \ is equal to  \ $ Y \setminus (Y_1 \cup Y_2).$
Indeed, for \ $(u,v,w) \in \Omega$, \ if \ $(x,y,z) \in \bar B_3(0,\varepsilon)$ \ verifies \ $F(x,y,z) = (u,v,w)$, we get \ $\varphi(v,w) = x^2 + y^2, \psi(v,w) = y.(e^x - 1)$ and \ $u = \psi(v,w) + \varphi(v,w) + z^2 - \varepsilon^2 $. \\
  \noindent One sees that \ $f(u,v,w) = (\varepsilon^2 - (x^2 + y^2 + z^2)).z^2$. 
  To finish, it is enough to note that the points of \ $F(\bar B_3(0,\varepsilon) \cap \{x.y = 0 \})$ \ and of \ $F(\partial B_3(0,\varepsilon))$ \ are never in \ $\Omega$. Hence \ $\tilde{f}$ \ is subanalytic.\\
   Let us show that it is continuous along \ $\partial \Omega$, since it is \ $\mathscr{C}^{\infty}$ \ on \ $\mathbb{R}^3 \setminus \partial \Omega$. Let \ $(u_0,v_0,w_0) \in \partial \Omega$. First assume that \ $(u_0,v_0,w_0)$ \ belongs to \ $\Delta^+ $. Then  \ $u_0 = \psi(v_0,w_0) $, in other words, we get the image by \ $F$ \ of a  point  \ $(x,y,z) \in \partial B_3(0,\varepsilon) \setminus \{ x.y = 0\}$. Hence the  limit of \ $ (u - \psi(v,w))$ \ when \ $(u,v,w) \in \Omega$ \ tends to \ $(u_0,v_0,w_0)$ \ is zero. As the functions \ $\psi$ \ and \ $\varphi$ \ are bounded on \ $\Omega$, the limit of \ $\tilde{f}$ \ is zero in such a point.\\
  If \ $(u_0,v_0,w_0) \in \Delta^-$, then we have the image of a  point in 
   $$(\bar B_3(0,\varepsilon)\cap \{z = 0\}) \setminus \{x.y = 0\}.$$
   Since the function \ $\psi$ \ is bounded on \ $\Gamma$ \ the limit of \ $f$ \ in such a point is zero, and so it is for \ $\tilde{f}$.\\
  If \ $(u_0,v_0,w_0) \in \Delta^0$ \ then we have \ $v_0 = w_0 = 0$ \ and the function \ $f$ \ is bounded, hence \ $\tilde{f}$ \ tends to \ $0$ \ in such a point.\\
  
  Let us finally show that  \ $\Omega$ \ is the set where \ $\tilde{f}$ \ is strictly positive. It is sufficient to check that \ $v.w \not= 0$ \ on \ $\Omega$. But  \ $v.w = 0$ \ entails \ $x.y = 0$ \ and so \ $v = w = 0$ \ and \ $u = x^2 + y^2 + z^2 - \varepsilon^2$, in other words, \ $(u,v,w) \in [-\varepsilon^2,0]\times (0,0) = \Delta^0$. Hence such \ $(v,w)$\ belongs to \ $\partial \Omega$.   
  \bigskip
  
  We have now constructed a subanalytic function \ $\tilde{f}$ \  on \ $\mathbb{R}^3$ \ which is continuous and strictly positive exactly on \ $\Omega \subset\subset \mathbb{R}^3$. By Corollary \ref{C2} there exists a positive integer \ $N$ \ such that \ $\tilde{f}^N$ \ is of class \ $\mathscr{C}^2$. Then one gets a Stein open subanalytic set of \ $\C^3$ \ which cuts \ $\mathbb{R}^3$ \ exactly on \ $\Omega$ \ as in the general proof of Theorem \ref{T:2}. \end{proof}

\subsection*{Acknowledgements}
The authors gratefully acknowledge the support of FCT and FEDER, within the
project ISFL-1-143 of the Centro de \'Algebra da Universidade de Lisboa.


\newpage
  
\begin{thebibliography}{15}
\bibitem{BM}
E. Bierstone and P. D. Milman, {\em Semi-analytic and subanalytic sets}, Publications Math\' ematiques de l'IH\' ES, \textbf{67}, 5-42 , (1988).
\bibitem{BDL}
J. Bolte, A. Danilidis and A. Lewis, {\em The {\L}ojasiewicz inequality for nonsmooth subanalytic functions with applications to subgradient dynamical systems}, Siam J. Optim.,  \textbf{17}, 4, 1205-1223, (2007).
\bibitem{MC}
M. Coste, {\em An introduction to o-minimal geometry},  Dip. Mat. Univ. Pisa, 
Dottorato di Ricerca in Matematica, Istituti Editoriali e Poligrafici Internazionali, Pisa (2000).
\bibitem{DM}
L. van den Dries and C. Miller, {\em Geometric categories and o-minimal structures}, 2001; revised version of the paper with the same name which appeared in Duke Math. J.,  497-540,  (1996).
\bibitem{G}
H. Grauert, {\em On Levi's Problem and the embedding of real analytic manifolds}, Annals of Math., \textbf{68}, 2, 460-472 (1958).

\bibitem{GR} R. Gunning and H. Rossi, {\em Analytic functions of several complex variables}, AMS Chelsea Publishing, (2009).
\bibitem{HH1}
H. Hironaka, {\em Subanalytic sets}, Algebraic Geometry and Commutative Algebra, Tokyo, Kinokuniya, 453-493, (1963).
\bibitem{HH2} H. Hironaka,
 {\em Introduction aux ensembles sousanalytiques}, Ast\'erisque (Soc. Math. France), \textbf{7-8}, 13-20, (1973).
 \bibitem{K} M. Kankaanrinta, {\em The simplicity of certain groups of subanalytic diffeomorphisms}, Differential Geometry and its Applications, 
 \textbf{27}, 661-670,  (2009).
 
 \bibitem{KK}
 K. Kurdyka, {\em Points r\'eguliers d'un sous-analytique}, Annales de l'Inst. Fourier, \textbf{38}, 1, 133-156, (1988).
 
 \bibitem{KS}
 M. Kashiwara and P. Schapira, {\em Ind-sheaves},  Ast\'erisque, Soc. Math. France, \textbf{271} (2001). 
 
 
\bibitem{L}
S. {\L}ojasiewicz,
 {\em Sur le probl\`eme de la division}, Studia Math., \textbf{8}, 87-136, (1959).
 \bibitem{S} M. Shiota, {\em Geometry of subanalytic and semialgebraic sets}, Progress in Mathematics, \textbf{150}, (1997).
 \bibitem{LP}
 L. Prelli,  {\em Sheaves on subanalytic sites}, Rendiconti del Seminario Matematico
dell'Universit\`a di Padova, \textbf{120},   (2008).

 
 \bibitem{T}
 M. Tamm, {\em Subanalytic sets in calculus of variations}, Acta Mathematica, \textbf{146}, 167-199 (1991).
 \bibitem{SV} W. Schmid and K. Vilonen, {\em Characteristic
cycles of constructible sheaves},  Invent. Math. \textbf{124}, 451--502 (1996).
\end{thebibliography}
\end{document}